\documentclass[letterpaper]{amsart}
\usepackage{amsmath}
\usepackage{amstext}
\usepackage{amssymb}
\usepackage{amsfonts}
\usepackage{enumerate}
\usepackage{braket}
\usepackage{color}
\usepackage[all]{xy}
\usepackage{url}
\usepackage{bbm}
\usepackage{mathtools}

\usepackage{hyperref,cleveref,graphics,mathrsfs}

\numberwithin{equation}{section}

\newtheoremstyle{note}
{1em}
{1em}
{}
{}
{\bfseries}
{: }
{.5em}
{}

\newtheorem{theorem}{Theorem}[section]
\newtheorem{lemma}[theorem]{Lemma}
\newtheorem{proposition}[theorem]{Proposition}
\newtheorem{corollary}[theorem]{Corollary}

\theoremstyle{note}
\newtheorem{remark}[theorem]{Remark}

\newtheorem{definition}[theorem]{Definition}

\newtheorem*{case1}{\textbf{Case 1}}
\newtheorem*{case2}{\textbf{Case 2}}
\newtheorem{claim}[theorem]{Claim}

\newcommand{\N}{{\mathbb{N}}}
\newcommand{\R}{{\mathbb{R}}}
\newcommand{\C}{{\mathbb{C}}}

\newcommand{\tn}[1]{{\left\vert\kern-0.25ex\left\vert\kern-0.25ex\left\vert #1 
    \right\vert\kern-0.25ex\right\vert\kern-0.25ex\right\vert}}

\newcommand{\OH}{\mathrm{OH}}
\newcommand{\eps}{\varepsilon}

\newcommand{\M}{{\mathrm{M}}}

\newcommand{\cZ}{{\mathcal{Z}}}

\newcommand{\vertiii}[1]{{\left\vert\kern-0.25ex\left\vert\kern-0.25ex\left\vert #1 
    \right\vert\kern-0.25ex\right\vert\kern-0.25ex\right\vert}}

\DeclareMathOperator{\cb}{cb}
\DeclareMathOperator{\CB}{CB}

\DeclareMathOperator{\supp}{supp}

\DeclareMathOperator{\spn}{span}

\DeclareMathOperator{\MIN}{MIN}
\DeclareMathOperator{\MAX}{MAX}

\title[Coarse geometry and  $\ell_1$ and $c_0$-sums of operator spaces]{Coarse geometry of operator spaces and complete isomorphic embeddings into  $\ell_1$ and $c_0$-sums of operator spaces  }

\author[B. M. Braga]{Bruno M. Braga}

\address[B. M. Braga]{PUC-Rio, Rua Marquês de São Vicente 225, Rio de Janeiro, RJ, Brazil.}
\email{demendoncabraga@gmail.com}
\urladdr{https://sites.google.com/site/demendoncabraga/}
\thanks{B. M. Braga was partially supported by NSF Grant   DMS 2054860. \\ T.~Oikhberg was partially supported by the NSF award 1912897.}

\author[T. Oikhberg]{Timur Oikhberg}

\address[T. Oikhberg]{Department of Mathematics, University of Illinois, Urbana IL 61801, USA}
\email{oikhberg@illinois.edu}
\urladdr{https://faculty.math.illinois.edu/~oikhberg/}


\begin{document}
 
\maketitle

 \begin{abstract}
The nonlinear geometry of operator spaces has recently started to be investigated. Many  notions of nonlinear embeddability have been introduced so far, but, as noticed before by other authors, it was not clear whether they could be considered  ``correct notions''. The main goal of these notes is to provide the missing evidence to support that \emph{almost complete coarse embeddability} is ``a correct notion''. This is done by proving results about the complete  isomorphic theory of $\ell_1$-sums of certain operators spaces. Several results on the complete isomorphic theory of $c_0$-sums of operator spaces are also obtained. 
 \end{abstract}


\section{Introduction}\label{SectionIntro}
This article concerns the   nonlinear theory of operator spaces, which   recently started to be investigated  (\cite{BragaChavezDominguez2020PAMS,BragaChavezDominguezSinclair,Braga2021OpSp}), as well as their  isomorphic theory.  We refer the reader to Section \ref{SectionPrelim} for the background in operator space theory needed for these notes. For now, we only recall some basics: an \emph{operator space} is a Banach space $X$ together with an isometric embedding into $B(H)$ ---  the space of bounded operators on some Hilbert space $H$. This isometric embedding allows us to see $X$ as a subspace of $B(H)$ and each $\M_n(X)$ as a subspace of $\M_n(B(H))\cong B(H^{\oplus n})$; where  $\M_n(X)$ denotes the space of $n$-by-$n$ matrices with entries in $X$ and $H^{\oplus n}$ denotes the $\ell_2$-sum of $n$ copies of $H$.   The operator norm of $ B(H^{\oplus n})$   induces a Banach norm on each $\M_n(X)$, which we denote by $\|\cdot \|_{\M_n(X)}$. Given another operator space $Y$,  a subset $A\subset X$, and a map $f\colon A\to Y$, the \emph{$n$-th amplification of $f$} is the map  $f_n\colon\M_n(A)\to \M_n(Y)$  given by  
\[f_n([x_{ij}])=[f(x_{ij})]\ \text{ for all }\ [x_{ij}]\in \M_n(X).\]
If $f$ is a linear operator,  so is each $f_n$ and we denote its operator norm by $\|f_n\|_n$. We then say that $f$ is \emph{completely bounded} if its \emph{$\cb$-norm} is finite, i.e., if 
\[\|f\|_{\cb}=\sup_{n\in\N}\|f_n\|_n<\infty.\]
Completely bounded maps then   induce the notions of complete isomorphisms and complete isomorphic embeddings between operator spaces.

\subsection{The nonlinear theory of operator spaces} Complete   boundedness  naturally inspire the following version of coarseness for nonlinear maps between operator spaces: a   map $f\colon X\to Y$ between operator spaces is \emph{completely coarse} if for all $r>0$ there is $s>0$ such that \[\|[x_{ij}]-[y_{ij}]\|_{\M_n(X)}\leq r\ \text{ implies }\ \|[f(x_{ij})]-[f(y_{ij})]\|_{\M_n(Y)}\leq s\]
for all $n\in\N$ and all $[x_{ij}], [y_{ij}]\in \M_n(X)$.\footnote{The definition of a \emph{coarse} map between Banach spaces is precisely this one for $n=1$.} However, as shown in   \cite[Theorem 1.1]{BragaChavezDominguez2020PAMS},  completely coarse maps are automatically $\R$-linear. Therefore, this notion does not lead to an interesting nonlinear theory of operator spaces, which has  led to   a search for   ``correct'' notions of nonlinear morphisms, embeddings, and equivalences  between operator spaces. In order to remedy this issue,  \cite{BragaChavezDominguez2020PAMS}  
  proposed to look at the coarse version of \emph{almost} complete isomorphic embeddings, instead of complete isomorphic embeddings.  
  Precisely:

 \begin{definition}(\cite[Definition 4.1]{BragaChavezDominguez2020PAMS})
 Let $X$ and $Y$ be operator spaces. 
 \begin{enumerate}
 \item A sequence $(f^n\colon X\to Y)_n$ of linear operators is called an \emph{almost complete isomorphic embedding} if the amplifications 
 \[\Big(f^n_n\colon {\M_n(X)}\to \M_n(Y)\Big)_{n}\] 
 are equi-isomorphic embeddings.\footnote{We call a sequence $(g_n:X_n\to Y_n)_n$ of isomorphic embeddings \emph{equi-isomorphic} if $\sup_n\|g_n\|,\sup_{n\in\N}\|g_n^{-1}\|<\infty$, where $g^{-1}_n$ is defined only on the image of $g_n$.}
 
 \item 
  A sequence $(f^n\colon X\to Y)_n$ of maps is called an \emph{almost complete coarse embedding} if the amplifications 
 \[\Big(f^n_n\colon {\M_n(X)}\to \M_n(Y)\Big)_{n}\] 
 are equi-coarse embeddings.\footnote{Recall, a family of maps $(f_n:X_n\to Y_n)$ between metric spaces $(X_n,d_n)$ and $(Y_n,\partial_n)$ are \emph{equi-coarse embeddings} if for all $r>0$ there is $s>0$ such that (1) $d_n(x,z)\leq r$ implies $\partial_n(f_n(x),f_n(z))\leq s$ and (2) $d_n(x,z)\geq s$ implies $\partial_n(f_n(x),f_n(z))\geq r$.} 
 \end{enumerate}
 \end{definition}

  The main results of \cite{BragaChavezDominguez2020PAMS,BragaChavezDominguezSinclair}
  show that almost complete coarse embeddability is strictly weaker than almost complete isomorphic embeddability but still strong enough to capture linear aspects of   operator space structures  (see, for instance, \cite[Theorem 1.2]{BragaChavezDominguez2020PAMS} and  \cite[Theorems 1.4 and 1.6]{BragaChavezDominguezSinclair}). 
  
  There is however a weakness in the methods of \cite{BragaChavezDominguezSinclair}. Precisely, although 
  \cite[Theorem 1.4]{BragaChavezDominguezSinclair} gives examples of operator spaces $X$ and $Y$ such that $X$ almost completely coarsely embeds into $Y$ but does not almost completely isomorphically embed into $Y$, the restriction for the latter happens already in the Banach level, i.e., $X$ does not even isomorphically embed into $Y$. The next definition was introduced by the first named author in \cite{Braga2021OpSp} precisely to fix this problem and it is the  notion of nonlinear embeddability which we deal with in these notes. Given a Banach space $X$, $B_X$ denotes its closed unit ball.

\begin{definition}(\cite[Definitions 1.2 and 1.6]{Braga2021OpSp})
Let $X$ and $Y$ be operator spaces. 
\begin{enumerate}
\item
We say that the \emph{bounded
subsets of $X$ almost completely coarsely embed into $Y$} if there is a sequence of maps
$(f^n \colon n \cdot B_X \to Y )_n$ such that the amplifications
\[\Big(f^n_n\restriction _{n\cdot B_{\M_n(X)}}\colon n\cdot B_{\M_n(X)}\to \M_n(Y)\Big)_{n}\]
are equi-coarse embeddings.\footnote{Notice that $B_{\M_n(X)}$ is contained  in $\M_n(B_X)$, so   $f^n_n\restriction _{n\cdot B_{\M_n(X)}}$ is well defined.}
\item We say that the \emph{bounded
subsets of $X$ and $Y$ are almost completely coarsely equivalent} if there is a sequence of bijections $(f^n \colon  X \to Y )_n$ such that, letting $g^n=(f^n)^{-1}$, the amplifications
\[\Big(f^n_n\restriction _{n\cdot B_{\M_n(X)}}\colon n\cdot B_{\M_n(X)}\to \M_n(Y)\Big)_{n}\text{ and }\Big(g^n_n\restriction _{n\cdot B_{\M_n(Y)}}\colon n\cdot B_{\M_n(Y)}\to \M_n(X)\Big)_{n}\]
are equi-coarse embeddings. 
\end{enumerate}
\end{definition}

Although these  notions of embeddability and equivalence are very weak, they are still strong enough to capture some of the linear aspects of operator spaces. For instance, if the bounded subsets of an operator space $X$ almost completely coarsely embed into G. Pisier's operator Hilbert space $\OH$, then $X$ completely isomorphically embeds into $\OH$ (\cite[Corollary 1.5]{Braga2021OpSp}). Also, 	if $R$ and $C$ are the row and the column operator spaces, respectively,\footnote{See Section \ref{SectionPrelim} for the precise definition of those operator spaces} and $(R,C)_\theta$ is the complex interpolation space with parameter $\theta$ (see \cite[Section 2.7]{Pisier-OS-book} for interpolation of operator spaces), where $\theta\in [0,1]$, then  \cite[Theorem 1.3]{Braga2021OpSp}  shows that the family $((R,C)_\theta)_{\theta\in [0,1]}$ is incomparable with respect to almost complete coarse embeddability of bounded subsets. 

On the other hand, this notion of embeddability is indeed  weaker than complete isomorphic embeddability and the reason for that does not occur in the Banach level.  Precisely:

\begin{theorem}\emph{(}\cite[Theorem 1.8]{Braga2021OpSp}\emph{)}
 There are operator spaces $X$ and $Y $ such that
 \begin{enumerate}
 \item  $X$ linearly isomorphically embeds into $Y$,
 \item
 $X$ does not completely isomorphically embed into $Y$, and
\item the bounded subsets of $X$ and $Y $ are almost completely coarsely equivalent.
 \end{enumerate}\label{ThmBragaOpSpIsrael}
\end{theorem}

We point out however that Theorem \ref{ThmBragaOpSpIsrael}  only  partially fixes the issue brought up above with the current state of the nonlinear theory of operator spaces. Indeed, if $X$ and $Y$ are given by Theorem \ref{ThmBragaOpSpIsrael}, then, while the bounded subsets of $X$ almost completely coarsely embed into $Y$, $X$ does not completely isomorphically embed into $Y$, and the reason for that does not happen in the Banach level, the methods in \cite{Braga2021OpSp} were not strong enough to guarantee that $X$ does not \emph{almost} completely isomorphically embed into $Y$ --- which would be a more desirable result since the bounded subsets of $X$ only \emph{almost} completely coarsely embed into $Y$.

The main goal of these notes is to resolve this issue. Philosophically speaking, this shows that the notion of almost completely coarse embeddability of subsets is ``a correct one'' --- we emphasise the  indefinite article here. With this interpretation   in mind, this paper finishes the work initiated in \cite{BragaChavezDominguez2020PAMS}, and continued in \cite{BragaChavezDominguezSinclair,Braga2021OpSp}, of showing that there is a genuinely  interesting and highly nontrivial nonlinear theory for operator spaces.

 \subsection{Embeddings into $\ell_1$-sums}
 We now describe our main result on the nonlinear theory of operator spaces. The following  strengthens  Theorem \ref{ThmBragaOpSpIsrael} and  answers the question asked in the paragraph following \cite[Theorem 1.8]{Braga2021OpSp}.  

\begin{theorem}\label{Thm.Main.Nonlinear}
There are separable operator spaces $X$ and $Y$ such that 
\begin{enumerate}
\item $X$ isomorphically embeds into $Y$,
\item $X$ does not almost completely isomorphically embed into $Y$, and
\item the bounded subsets of $X$ and $Y$ are almost completely coarsely equivalent.
\end{enumerate}
\end{theorem}

We point out that Theorem \ref{Thm.Main.Nonlinear} strengthens Theorem \ref{ThmBragaOpSpIsrael} not only since $X$ does not \emph{almost} completely isomorphically embed into $Y$, but also because the spaces given by Theorem \ref{Thm.Main.Nonlinear} are separable, which is not the case in Theorem \ref{ThmBragaOpSpIsrael}. 

We  briefly describe our approach to  Theorem \ref{Thm.Main.Nonlinear}. We prove this   theorem  by studying the complete isomorphic theory of $\ell_1$-sums of operator spaces.  Precisely, let $Q:\MAX(L_1)\to \MIN(\ell_2)$ be a completely bounded surjection given by the composition of surjections $L_1\to \ell_1$ and $\ell_1\to \ell_2$. This allow us to define operator spaces $(Y_i)_{i\in\N}$ such that the norm of each $Y_i$ is given by 
\[\|y\|_{Y_i}=\max\{\|y\|,2^i\|Q(y)\|\}\]
(see Section \ref{Secl1sum} for details). Our main technical result of Section \ref{Secl1sum} then shows that $\MIN(\ell_2)$ cannot almost completey isomorphicaly embed into $(\bigoplus_iY_i)_{c_0}$. This proof is quite technical and it is done in  Lemma \ref{Lemma1} below. Together with \cite[Theorem 4.3]{Braga2021OpSp} (restated below as Theorem \ref{ThmBraga021OpSp}), this will allow us to obtain Theorem \ref{Thm.Main.Nonlinear}.
 
\subsection{Embeddings into $c_0$-sums}
In the second part of this paper, we leave the nonlinear theory aside and  move to study the complete isomorphic embeddability of certain operator spaces  into certain $c_0$-sums $(\bigoplus_n Y_n)_{c_0}$. Precisely, Section \ref{Secc0sum} deals with operator space versions (and counterexamples) of the following classic result from Banach theory, whose prove follows from a standard gliding hump argument (cf. \cite[Proposition 2.c.4]{LT1} or Proposition \ref{compact operator into} below). 


\begin{proposition}[Folklore]
Suppose $X$ and $(Y_i)_{i\in\N}$ are Banach spaces and assume that  $X$ is infinite dimensional and does not contain an isomorphic copy of $c_0$. If $X$ isomorphically embeds into $(\bigoplus_i Y_i)_{c_0}$, then some infinite dimensional subspace of $X$ isomorphically embeds into some $Y_i$.\label{no embedding}
\end{proposition}

We show that \Cref{no embedding} 
has an operator space version for  homogeneous Hilbertian operator spaces as long as we restrict ourselves to $c_0$-sums of a single operator  space.\footnote{Recall, an operator space $X$ is \emph{Hilbertian} if it is isomorphic (as a Banach space) to $\ell_2$. Also, given $\lambda\geq 1$, $X$ is \emph{$\lambda$-homogeneous} if $\|u\|_{\cb}\leq \lambda\|u\|$ for all operators $u\colon X\to X$. We then say $X$ is \emph{homogeneous} if it is $\lambda$-homogeneous for some $\lambda\geq 1$.} Precisely:

\begin{theorem}\label{Thm.HomHilb.SingleSpace}
Let $X$ and $Y$ be  homogeneous Hilbertian spaces and assume that $X$ has infinite dimension. If $X$  completely isomorphically embeds into $c_0(Y)$, then $X$ completely isomorphically embeds  into $Y$. In particular, if $X$ and $Y$ have the same density character, then $X$ and $Y$ are completely isomorphic.
\end{theorem}

The restriction in Theorem \ref{Thm.HomHilb.SingleSpace} of only considering $c_0$-sums of a single operator space is not superfluous. Precisely, we show the following: 

\begin{theorem}\label{Thm.Prop.HomHilb.MultipleSpace}
Let $(Y_i)_{i\in\N}$ be operator spaces all of which are completely isomorphic to $\MIN(\ell_2)$. There exists a separable homogeneous Hilbertian operator space $X$ such that $X$ completely isomorphically embeds into $(\bigoplus_nY_n)_{c_0}$, but  $X$ is  not completely isomorphic to $\MIN(\ell_2)$.
\end{theorem}

Notice that, since the operator space $X$ in Theorem \ref{Thm.Prop.HomHilb.MultipleSpace} is homogeneous and Hilbertian, the conclusion of this theorem implies that no infinite subspace of $X$ completely isomorphically embeds into any of the $Y_i$'s. Hence,  Theorem \ref{Thm.Prop.HomHilb.MultipleSpace} does indeed show that the operator space version of Proposition \ref{no embedding} does not hold in general.

Although Theorem \ref{Thm.Prop.HomHilb.MultipleSpace} says that we cannot generalize Theorem \ref{Thm.HomHilb.SingleSpace} to arbitrary $c_0$-sums, we show that this can be done at least for some specific operator spaces $X$. In the next theorem, $R$ and $C$ denote the \emph{row} and \emph{column operator spaces}, respectively, and  $R\cap C$ their \emph{intersection operator space} (see Section \ref{SectionPrelim} for precise definitions).

\begin{theorem}\label{Thm.fix problem}
Let $(Y_i)_{i\in\N}$ be operator spaces all of which are completely isomorphic to $\MIN(\ell_2)$.  If  $X\in \{R,C,R\cap C, \MAX(\ell_2)\}$, then  $X$ does not completely isomorphic embeds into $(\bigoplus_iY_i)_{c_0}$.
 \end{theorem}

 Finally, we observe that, if, in the setting of Theorem \ref{Thm.Prop.HomHilb.MultipleSpace}, the homogeneity of $X$ is not assumed, then we can embed an ``unexpected'' space $X$ not just into a $c_0$-sum of $Y_i$'s, but also into the simpler space $c_0(Y)$:
 
 \begin{theorem}\label{Thm.Counterexample} 
There are separable Hilbertian operator spaces $X$ and $Y$ such that $X$ completely isometrically embeds into $c_0(Y)$, but   $X$ does not completely isomorphically embed  into $Y$. 
 \end{theorem}
 
  We point that, although $X$ does not completely isomorphically embeds into $Y$ in the theorem above, our example produces an $X$ with infinite dimensional subspaces which do completely isomorphically embed into $Y$.


\section{Preliminaries}\label{SectionPrelim}
 In this section, we recall the basics of operator space theory which will be used throughout the paper.    We refer the reader to \cite{Pisier-OS-book} for a monograph on  this theme. We start by saying that, for each $k\in\N$,   $\M_k$ denotes the space of $k$-by-$k$ matrix with complex entries, i.e., $\M_k=\M_k(\C)$.
 
 Let $(X_\lambda)_{\lambda \in \Lambda}$ be a family   of operator spaces. Then $(\bigoplus_\lambda X_\lambda)_{\ell_\infty}$ denotes the \emph{$\ell_\infty$-sum of $(X_\lambda)_{\lambda\in \Lambda}$}, i.e., \[\Big(\bigoplus_\lambda X_\lambda \Big)_{\ell_\infty}=\Big\{(x_\lambda)_\lambda\in X^\Lambda\mid \sup_{\lambda\in \Lambda}\|x_\lambda\|<\infty \Big\}\]  together with the operator space structure given by 
\[\|[x_{ij}]\|_{\M_k((\bigoplus_{\lambda}X_\lambda)_{\ell_\infty})}=\sup_{\lambda\in \Lambda}\|[x_{ij}(\lambda)]\|_{\M_k(X_\lambda)}\]
for all $k\in\N$ and all $ [x_{ij}]=([x_{ij}(\lambda)])_\lambda\in \M_k((\bigoplus_\lambda X_\lambda)_{\ell_\infty})$. If $\Lambda=\N$, the \emph{$c_0$-sum of $(X_n)_n$}, denoted by $(\bigoplus_n X_n  )_{c_0}$,  is the operator subspace of $(\bigoplus_n X_n  )_{\ell_\infty}$ consisting of all $(x_n)_n$ such that $\lim_n\|x_n\|=0$. If all $X_n$'s are the same, say $X=X_n$ for all $n\in\N$, we simply write $c_0(X)$ for $(\bigoplus_nX_n)_{c_0}$. Also, if the sequence $(X_n)_n$ is finite, say $X_1,\ldots, X_k$, we simply write $X_1\oplus_\infty\ldots\oplus_\infty X_k$ for their $c_0$-sum.

Similarly,  $(\bigoplus_nX_n)_{\ell_1}$ denotes the \emph{$\ell_1$-sum of $(X_n)_n$}, i.e., \[\Big(\bigoplus_nX_n\Big)_{\ell_1}=\Big\{(x_n)_n\in X^\N\mid \sum_{n\in\N}\|x_n\|<\infty\Big\}\]  together with the Banach norm $\|(x_n)_n\|_{ (\bigoplus_nX_n )_{\ell_1}}=\sum_{n\in\N}\|x_n\|$ and the operator space structure given by the isometric embedding   \[J\colon \Big(\bigoplus_nX_n\Big)_{\ell_1}\to \Big(\bigoplus_{u\in P}B(H_u)\Big)_{\ell_\infty},\] 
 where $P$ denotes the family of all sequences $u=(u_n)_n$ of completely contractive\footnote{Recall, an operator $u:X\to Y$ between operator spaces is \emph{completely contractive} if $\|u\|_{\cb} \leq 1$.} maps $u_n:X_n\to B(H_u)$  and $J((x_n)_n)=(u_n(x_n))_n$ for all $(x_n)_n\in (\bigoplus_nX_n)_{\ell_1}$ (we can restrict ourselves to the Hilbert spaces $H_u$ whose density chracter does not exceed that of $X$).  If all $X_n$'s are the same, say $X=X_n$ for all $n\in\N$, we simply write $\ell_1(X)$ for $(\bigoplus_nX_n)_{\ell_1}$. If the sequence $(X_n)_n$ is finite, say $X_1,\ldots, X_k$, we simply write $X_1\oplus_1\ldots\oplus_1 X_k$ for their $\ell_1$-sum (see
 \cite[Section 2.6]{Pisier-OS-book}  for details on direct sums  of operator spaces).
 
Given $\lambda\geq 1$,  an operator space $X$ is \emph{$\lambda$-minimal} if for any operator space $Y$ and any operator $u:Y\to X$, we have that $\|u\|_{\cb}\leq \lambda \|u\|$. We say that $X$ is \emph{minimal} if it is $\lambda$-minimal for some $\lambda\geq 1$.\footnote{ We point out that authors interested in the isometric theory of operator spaces often use the term \emph{minimal} to refer to what we are calling a \emph{$1$-minimal} operator space.}

Let $X$ be a Banach space. Then $X$ isometrically embeds into $C(B_{X^*})$ --- the Banach space of continuous functions on the closed unit ball of $B_{X^*}$. Since $C(B_{X^*})$ is a Banach algebra, we can see it as a $\mathrm{C}^*$-subalgebra of  $B(H)$, for some Hilbert space $H$. This gives a canonical operator space structure on $X$ and we denote this operator space by $\MIN(X)$. On the other hand, $X$ can be endowed with the operator space structure given by the isometric embedding 
\[ X\to \Big(\bigoplus_{u\in P} B(H_u)\Big)_{c_0},\]
where $P$ denotes the family of all contractions $u:X\to B(H_u)$ and $J(x)=(u(x))_{u\in P}$. We denote this operator space  by $\MAX(X)$. 
 
Throughout these notes, $R$ and $C$ denote the \emph{row} and the \emph{column operator spaces}, respectively. I.e., let $(e_i)_i$ denote the canonical orthonormal basis of $\ell_2$ and let $(e_{ij})_{i,j\in\N}$ denote the \emph{matrix units} --- that is, the operators in $B(\ell_2)$ such that $e_{ij}e_j=e_i$ and $e_{ij}e_k=0$ for all $i,j,k\in \N$ with $k\neq j$. Then
\[R=\overline{\mathrm{span}}\{e_{1i}\mid i\in\N\}\ \text{ and }\ C=\overline{\mathrm{span}}\{e_{i1}\mid i\in\N\}.\]
Clearly, both $R$ and $C$ are isometric to $\ell_2$; this can be seen since the maps  $r\colon \ell_2\to R$ and $c\colon\ell_2\to C$ determined by $r(e_i)=e_{1i}$ and $c(e_i)=e_{i1}$, for all $i\in\N$, are isometries.   The operator space $R\cap C$ is the Banach space $\ell_2$ together with the operator space structure given by the isometric embedding
\[x\in \ell_2\mapsto (r(x),c(x))\in R\oplus_\infty C.\]
At last, $R+C$ denotes the operator space given by the quotient $R\oplus_1 C/\Delta$, where $\Delta=\{(r(x),-c(x))\mid x\in \ell_2\}$ (see \cite[Page 194]{Pisier-OS-book}).

\section{Embeddings into certain $\ell_1$-sums}
\label{Secl1sum} 
 The main goal of this section is to prove Theorem \ref{Thm.Main.Nonlinear}. For that, we must recall some results obtained in \cite{Braga2021OpSp}. Let $Y$ and $X$ be operator spaces and consider a completely bounded map $Q\colon Y\to X$ which is also a Banach quotient map. For each $m\in\N$, let $Y_m$ be the Banach space such that $Y_m=Y$ as a vector space  and with norm given by 
 \[\|y\|_{Y_m}=\max\{ \|y\|,2^m\|Q(y)\|\} \]
 for all $y\in Y_m$. Moreover, endow $Y_m$ with the operator space structure given by
 \[  \|[y_{ij}]\|_{\M_k(Y_m)}=\max\{ \|[y_{ij}]\|_{\M_k(Y)},2^m\|[Q(y_{ij})]\|_{\M_k(X)}\}\]
 for all $k\in\N$ and all $[y_{ij}]\in \M_k(Y_m)$. It follows from Ruan's Theorem that this indeed induces an operator space structure on each $Y_m$ (\cite[Section 2.2]{Pisier-OS-book}). Notice also that  each $Y_m$ is completely isomorphic to $Y$.

Let $\cZ(Q)=(\bigoplus_mY_m)_{\ell_1}$. Then, by the universal properties of $\ell_1$-directed sums of operator spaces, there is a completely bounded map 
\[\tilde Q\colon \cZ(Q)\to X\]
 such that $\tilde Q\circ i_m= Q$ for all $m\in\N$; where each $i_m\colon Y_m\hookrightarrow \cZ(Q)$ denotes the canonical inclusion (\cite[Subsection 1.4.13]{BlecherLeMerdy2004}). Clearly, $\tilde Q$ is also a Banach quotient map. 
 
We can now state one of the main technical results from \cite{Braga2021OpSp}.
 
 \begin{theorem} \emph{(}\cite[Theorem 4.3]{Braga2021OpSp}\emph{)}
 Let $X$ and $Y$ be operator spaces, $Q\colon  Y \to X $ be a completely bounded
map which is also a Banach quotient, and let $\tilde Q \colon  \cZ(Q) \to X$ be as above. Then, the
bounded subsets of $\cZ(Q)$ and $X\oplus  \ker(\tilde Q)$ are almost completely coarsely 
equivalent.\label{ThmBraga021OpSp}
 \end{theorem}

  \begin{remark}\label{RemarkDif1}
  We point out to the reader that, strictly speaking, the operator spaces $Y_m$ were defined in \cite{Braga2021OpSp} to have norm 
 \[ \vertiii {y}_{Y_m}=\max\{2^{-m} \|y\|,\|Q(y)\|\}. \]
 This is however just a formal difference since $(Y_m,\|\cdot \|_{Y_m})$ and $(Y_m,\vertiii{\cdot}_{Y_m})$ are clearly completely isometric to each other.  So, Theorem \ref{ThmBraga021OpSp} remains valid under this slight change in the definition of the $(Y_m)_m$'s. 
  \end{remark}

 We will prove Theorem \ref{Thm.Main.Nonlinear}  by looking at an appropriate quotient map $Q\colon Y\to X$.  Suppose  $Y=\MAX( L_1)$ and $X=\MIN(\ell_2)$. Since every separable Banach space is a quotient of $\ell_1$ (\cite[Theorem  2.3.1]{AlbiacKaltonBook}) and $\ell_1$ embeds into $L_1$ complementably (\cite[Proposition 5.7.2]{AlbiacKaltonBook}),   there is a completely bounded   map $Q\colon Y\to X$ which is also a Banach quotient.

 \begin{lemma}\label{Lemma1}
 Let $Y=\MAX( L_1)$ and  $X=\MIN(\ell_2)$, and let  $Q\colon Y\to X$ be the map described above. Let $(Y_m)_m$ be as defined above for $Y$. Then $X$ does not almost completely isomorphically embed into $\cZ(Q)=(\bigoplus_m Y_m)_{\ell_1}$.
 \end{lemma} 
 
 \begin{proof}
We start by noticing that $\cZ(Q)$ can be viewed as being embedded into $\ell_1(Y)\oplus_\infty \ell_1(X)$.  Indeed, for each $m\in\N$, let $J(m)\colon Y_m\to Y\oplus_\infty X$ be the map given by 
\[J(m)(y)=( y,2^mQ(y))\]
  for all $y\in Y_m$. So, each $J(m)$ is a complete isometric embedding and this allows us to view each $Y_m$ as a subspace of $Y\oplus_\infty X$. Consequently, we can view $\cZ(Q)$ as a subspace of $\ell_1(Y\oplus_\infty X)$. Since $\ell_1(Y\oplus_\infty X)$ is  completely isomorphic to $\ell_1(Y)\oplus_\infty \ell_1(X)$, we can view $\cZ(Q)$ as being embedded in it. Let \[P_1\colon \ell_1(Y)\oplus_\infty \ell_1(X)\to \ell_1(Y)\ \text{ and }\  P_2\colon \ell_1(Y)\oplus_\infty \ell_1(X)\to \ell_1(X)\] denote the standard projections.

 \begin{claim}\label{l:ignore quotients}
Given any linear map $u =\oplus_mu(m)\colon  X \to \cZ(Q)$, there exists an orthonormal sequence $(\xi_i)_i$ in $ X$ such that $\lim_i \sum_m 2^m \|Q u(m) \xi_i\| = 0$. Consequently, $\lim_i \|P_2 u \xi_i\| = 0$.
 \end{claim}
 
 \begin{proof}
We start by setting some notation. For each $m\in\N$, let $v(m)  = 2^m Q u(m)$. As $X=\ell_2$ as a Banach space, we have that  $v(m)\in B(\ell_2)$. As $u$ takes values in $\cZ(Q)$, the definition of the norm in this space implies that  $v = \oplus_m v(m) \in B(\ell_2, \ell_1(\ell_2))$. Finally, for each $I\subset  \N$, let   $v(I)= \oplus_{m\in I} v(m)$ and $v(I)^\perp = \oplus_{m\not\in I}  v(m)$.  We now  point out the following basic facts:
 
 \begin{itemize}
 \item For any finite dimensional $E \subset \ell_2$ and any $\varepsilon > 0$, there exists $k \in \N$ such that $\|v({\{1,\ldots, k\}})^\perp \xi\| \leq \varepsilon \|\xi\|$ for any $\xi \in E$. 
 
 \item  As $Q$ is the composition of an operator $L_1\to\ell_1$ with an operator $\ell_1\to \ell_2$ and  any operator $\ell_1\to \ell_2$ is strictly singular\footnote{Recall, an operator $u:X\to Y$ between Banach spaces is \emph{strictly singular} if none of its restrictions to an infinite dimensional subspaces of $X$ is an isomorphic embedding.  For more on strictly singular operators, see \cite[Section 2.c]{LT1}.}   (\cite[Theorem 2.1.9]{AlbiacKaltonBook}), we have that each $v(n)$ is strictly singular.  Consequently, each $v(I)$ is strictly singular for any finite $I\subset \N$. 
 \end{itemize}
 
Fix $\varepsilon > 0$. Notice that for any finite dimensional $E \subset \ell_2$, there exists an infinite dimensional $F \subset E^\perp$ such that $\|v\xi\| \leq \varepsilon \|\xi\|$ for any $\xi \in F$. Indeed, suppose this statement is false. Then, since each $v(I)$, for $I\subset \N$ finite, is strictly singular, we can find  sequences $(I_i)_i$ and $(\xi_i)_i$ such that 

\begin{enumerate}
\item each $I_i$ is a finite subset of $ \N$ and $\max (I_i)<\min(I_{i+1})$ for all $i\in\N$, 
\item   $(\xi_i)_i$ is a normalized  sequence in $\ell_2$ equivalent to its standard unit basis and such that   $\|v\xi_i\|\geq \eps$ for all $i\in\N$, and 
\item $\|v\xi_i-v({I_i})\xi_i\|\leq 2^{-i}$ for all $i\in\N$.
\end{enumerate}
Therefore, up to a constant $C>0$ independent on $k$, we have that \[\Big\|v\Big(\sum_{i=1}^k\xi_i\Big)\Big\|\geq Ck\] for all $k\in\N$. On the other hand, since  $(\xi_i)_i$ is   equivalent to its standard unit basis of $\ell_2$, we have that, \[\Big\|\sum_{i=1}^k\xi_i\Big\|\leq Dk^{1/2}\]  for all $k\in\N$, where   $D>0 $ is another constant independent on $k$. This gives us a contradiction since $v$ is bounded.

 We now construct the  required orthonormal sequence $(\xi_i)_i$ recursively as follows.  Pick a  norm one $\xi_1\in \ell_2=X$ with finite support (with respect to the canonical basis) and such that $\|v\xi_1\|\leq 2^{-1}$. Suppose finitely supported normalized vectors $\xi_1, \ldots, \xi_i\in \ell_2$ have been chosen so that 
 \begin{enumerate}
  \item $\supp(\xi_j)<\supp(\xi_{j+1})$ for all $j\in \{1,\ldots, i-1\}$, and 
  \item  $\|v\xi_j\|\leq 2^{-j}$ for all $j\in \{1,\ldots, i\}$.
\end{enumerate}    By the previous paragraph, we can choose a norm one finitely supported $\xi_{n+1}\in \ell_2$ such that  $\supp(\xi_{i+1})>\supp(\xi_{i})$ and   $\|v \xi_{i+1}\| < 2^{-i-1}$. It follows straightforwardly  from the definition of $v$ and the norm in $\ell_1(\ell_2)$ that $\lim_i \sum_m 2^m \|Q u(m) \xi_i\| = 0$.
 \end{proof}

 \begin{claim}\label{p:lower bound}
 For every $\gamma > 0$ there exists $m \in \N$ such that any operator $u \colon  X \to \cZ(Q)$ with $\|u^{-1}\| \leq 1$ satisfies $\|u_m\| \geq \gamma$.
 \end{claim}

 \begin{proof} 
Fix $\gamma>0$. Fix $u\colon X\to \cZ(Q)$ with $\|u^{-1}\|\leq 1$ --- we will find $m\in\N$ below which does not depend on $u$.   By Claim \ref{l:ignore quotients}, there exists an orthonormal sequence $(\xi_i)_i$ in $ \ell_2$ such that $\|P_2 u \xi_i\| <  2^{-i}$ for any $i\in\N$. Therefore,   \[ \|P_1u \xi_i\|\geq   \|u ^{-1}\| \|P_1u \xi_i\| = \|u ^{-1}\|\|u \xi_i\|\geq\|\xi_i\|\geq 1\]
for all $i\in\N$. 
 
 Notice that the range of $P_1$ can be identified with $\ell_1(L_1) = L_1(\N \times (0,1))$, endowed  with its natural maximal operator space structure, where   $\N\times (0,1)$ is considered together with its canonical measure, which we denote by	 $\mu$. As usual, we write $L_1(\mu)=L_1(\N \times (0,1))$.
 
Let $c > 0$ be a constant to be determined later. Since $R+C$ is not a minimal operator space, the identity $\MIN(\ell_2)\to R+C$ is not completely bounded. Hence,  there are   $m,N\in\N$, and $a_1, \ldots, a_N \in \M_m$, such that \[\Big\|\sum_i a_i \otimes \xi_i\Big\|_{\M_m \otimes \MIN(\ell_2)} = 1\ \text{ and }\ \Big\|\sum_i a_i \otimes \xi_i\Big\|_{\M_m \otimes (R+C)} > c\] 
(by \cite{Pisier-OS-book}, one can take $N \sim c^2$). Let ${\mathbb{T}}$ denote the unit torus, i.e., $\mathbb T=\{z\in \C\mid |z|=1\}$. Then, by the $1$-homogeneity of $\MIN(\ell_2)$,  we have that 
\begin{align*}
 \Big\|\sum_i \omega_i a_i \otimes \xi_i\Big\|_{\M_m \otimes \MIN(\ell_2)} &= \Big\|\sum_i a_i \otimes \omega_i \xi_i\Big\|_{\M_m \otimes \MIN(\ell_2)}\\
 & = \Big\|\sum_i a_i \otimes \xi_i\Big\|_{\M_m \otimes \MIN(\ell_2)} = 1
\end{align*}
for all $\omega = (\omega_i)_{i=1}^N \in {\mathbb{T}}^N$.   We will now show that there exists   $\omega = (\omega_i)_{i=1}^N \in {\mathbb{T}}^N$ such that \[\Big\|\sum_{i=1}^N \omega_i a_i \otimes P_1 u \xi_i\Big\|_{\M_m \otimes L_1(\mu)} \geq \kappa c,\] where $\kappa$ is a universal constant. From this inequality, it follows that $\|u_m\| \geq \kappa c$. Therefore,  taking $c = \gamma/\kappa$,  the proof of the proposition will be completed. Indeed, notice that $m\in\N$ depends only on $c$, therefore, since $\kappa$ is a universal constant, $m$ depends only on $\gamma$ and not on $u$.

 Let $\nu$ be the rotation invariant probability measure on ${\mathbb{T}}^N$.  For each $i\in \{1,\ldots, N\}$, let  $\phi_i\colon \mathbb T^N\to \mathbb T$ be the canonical projection onto the  $i$-th coordinate of $\mathbb T^N$. 
 We claim that $(\phi_i \otimes P_1 u \xi_i)_{i=1}^N$ is a \emph{$1$-completely unconditional basic sequence} in $L_1(\nu \otimes \mu)$ --- that is, for any sequence $(\alpha_i)_{i=1}^N$ of complex numbers, with $\max_i |\alpha_i| \leq 1$, the ``diagonal'' map
 $ \Phi_\alpha$ on $F = \spn[ \phi_i \otimes P_1 u \xi_i : 1 \leq i \leq N]$, taking $\phi_i \otimes P_1 u \xi_i$ to $\alpha_i \phi_i \otimes P_1 u \xi_i$, is completely contractive. By convexity, it suffices to show this holds if $|\alpha_i| = 1$ for every $i\leq N$.

To evaluate $\|\Phi_\alpha\|_{\cb}$ in this setting, it is more convenient to work not with the injective tensor product with matrix spaces $\M_n$, but rather, to consider the ``dual'' setting.
 For each $n\in\N$, $S_1^n$ denotes the operator space of  $n\times n$ trace class operators and, if $E$ is another operator space, then $S^n_1[E]$ denotes the projective operator space tensor product $S^n_1\otimes ^{ \wedge} E$. The reader is referred to \cite[Chapter 7]{ER}, \cite{Pisier-Asterisque}, or \cite[Chapter 4]{Pisier-OS-book} for a detailed treatment of this tensor product. 
 Here we mention two properties important for us.
 \begin{enumerate}
     \item For any map $T : E \to F$ between operator spaces, $\|T\|_{\cb} = \sup_n \|I \otimes T : S^n_1[E] \to S^n_1[F]$ (this follows from \cite[Lemma 1.7]{Pisier-Asterisque}, or by duality from \cite[Proposition 7.1.6]{ER}).
     \item  For any $\sigma$-finite measure $\mu$, $S^n_1[L_1(\mu)]$ is isometrically identified with $L_1(\mu, S^n_1)$ (see \cite[Proposition 2.1]{Pisier-Asterisque}).
 \end{enumerate}

 %
 Notice that,   if $(y_i)_{i=1}^N$ is in $ S^n_1$, then
 \begin{align*}
\Big \|\sum_{i=1}^N y_i \otimes \phi_i \otimes P_1 u \xi_i\Big\|_{S^n_1[L_1(\nu \otimes \mu)]} &
= \int \Big\|\sum_{i=1}^N y_i \otimes \phi_i(\omega) P_1 u \xi_i\Big\|_{L_1(\mu, S^n_1)} \, d\nu(\omega) 
\\ & = \int \Big\|\sum_{i=1}^N y_i \otimes \omega_i P_1 u \xi_i\Big\|_{L_1(\mu, S^n_1)} \, d\nu(\omega) .
 \end{align*}
 Now suppose $|\alpha_i| = 1$ for $1 \leq i \leq N$. By the rotation invariance of $\nu$, the right hand side equals
 \begin{align*}
 \int \Big\|\sum_{i=1}^N y_i \otimes \alpha_i \omega_i P_1 u \xi_i\Big\|_{L_1(\mu, S^n_1)} \, d\nu(\omega)
 & = \Big \|\sum_{i=1}^N y_i \otimes \alpha_i \phi_i \otimes P_1 u \xi_i\Big\|_{S^n_1[L_1(\nu \otimes \mu)]}
 \\ & = \Big \|(I \otimes \Phi_\alpha) \sum_{i=1}^N y_i \otimes \phi_i \otimes P_1 u \xi_i\Big\|_{S^n_1[L_1(\nu \otimes \mu)]}. 
 \end{align*}
Thus, $\Phi_\alpha$ is a complete isometry on $\spn[ \phi_i \otimes P_1 u \xi_i : 1 \leq i \leq N]$. This establishes the desired unconditionality.

 Notice that \[\|\phi_i \otimes P_1 u\xi_i\|_{L_1(\nu\otimes  \mu)} = \|P_1 u\xi_i\|_{L_1(\mu)}\geq 1\]
 for all $i\in\{1,\ldots, N\}$. Hence,  \cite[Proposition 4.3]{Oikhberg2004} gives  a constant $\kappa>0$ (independent on $u$ and $N$) such that  the operator
$
 T\colon F \to R+C $ determined by $ \phi_i \otimes P_1 u\xi_i \mapsto   e_i$, for each $i\in\N$,  has c.b.~norm at most $1/\kappa$; here $(e_i)_i$ denotes the canonical basis of $R+C$.  As noted above, for every $n$ we have
 $$
 \| {\mathrm{Id}} \otimes T : S_1^n[F] \to S_1^n[R+C]\| \leq \|T\|_{\cb} \leq \frac1\kappa .
 $$

As $\|\sum_i a_i \otimes \xi_i \|_{\M_m \otimes (R+C)} > c$,  it follows from the $1$-homogeneity of $R+C$ that    $\|\sum_i a_i \otimes e_i \|_{\M_m \otimes (R+C)} > c$. Therefore, \cite[Lemma 1.7]{Pisier-Asterisque} gives us  $m \times m$ matrices $b_1$ and $ b_2$ of Hilbert-Schmidt norm one, such  that \[\Big\|\sum_i b_1 a_i b_2 \otimes e_i\Big\|_{S^m_1[R+C]} > c.\]  
 By $1$-homogeneity of $R+C$ again, the same inequality holds if each $a_i$ above is replaced by $\omega_i a_i$, where  $(\omega_i)_{i=1}^N$ is an arbitrary element of ${\mathbb{T}}^N$. Consequently, 
 \begin{align*}
 &
  \int \Big\| \sum_i \omega_i b_1 a_i b_2 \otimes P_1 u \xi_i\Big\|_{S^m_1[L_1(\mu)]} \, d \nu(\omega)
  \\
 &
 =
 \Big\| \sum_i b_1 a_i b_2 \otimes \phi_i  \otimes P_1 u \xi_i\Big\|_{S^m_1[L_1(\nu \otimes \mu)]} 
 \\
&\geq  \|\mathrm{Id}\otimes T\|^{-1}  \Big\| \sum_i b_1 a_i b_2 \otimes    e_i\Big\|_{S^m_1[R+C]}    \\ 
  &>\kappa c .
\end{align*}
 Thus, there exists $\omega = (\omega_i)_{i=1}^N \in {\mathbb{T}}^N$ such that \[\Big\| \sum_i \omega_i b_1 a_i b_2 \otimes P_1 u \xi_i\Big\|_{S^m_1[L_1(\mu)]} \geq \kappa c.\] Applying \cite[Lemma 1.7]{Pisier-Asterisque} again, we conclude that \[\Big\|\sum_i \omega_i a_i \otimes P_1 u \xi_i\Big\|_{\M_m \otimes L_1(\mu)} \geq \kappa c\]
 and, by the discussion above, this proves the claim.
 \end{proof}
 
 Claim \ref{p:lower bound} immediately implies that there is no almost completely isomorphic embedding of $X$ into $\cZ(Q)$; so we are done. 
  \end{proof}
  
  \begin{proof}[Proof of Theorem \ref{Thm.Main.Nonlinear}]
   Let $Y=\MAX( L_1)$ and $X=\MIN(\ell_2)$. Let $Q\colon Y\to X$ be the quotient map described before Lemma \ref{Lemma1}.   Let $(Y_m)_m$, $\cZ(Q)$ and $\tilde Q$ be as above. By Theorem \ref{ThmBraga021OpSp},  the
bounded subsets of $\cZ(Q)$ and $X\oplus  \ker(\tilde Q)$ are almost completely coarsely equivalent.

Note that the formal identity between $(\oplus_{m \geq 1} Y_m)_{\ell_1}$ and $(\oplus_{m \geq 2} Y_m)_{\ell_1}$ is a (complete) isomorphism. Consequently, $\cZ(Q)$ is (completely) isomorphic to $Y \oplus \cZ(Q)$. 

It is well known that $\ell_2$ isometrically embeds into $L_1$ (\cite[Theorem 6.4.17]{AlbiacKaltonBook}). So, in our case,  $X\oplus \ker(\tilde Q)$ isomorphically embeds into $\cZ(Q)$. We are left to notice that $X\oplus \ker(\tilde Q)$ does not almost completely  isomorphically embed into $\cZ(Q)$. For that, it is enough to show that $X$ does not almost completely  isomorphically embed into $\cZ(Q)$.  This is precisely Lemma \ref{Lemma1}, so we are done. 
  \end{proof}
  
 \section{Embeddings into certain $c_0$-sums}\label{Secc0sum}
 
In this section, we leave the nonlinear theory aside and concentrate fully on the completely isomorphic theory of $c_0$-sums of operator spaces. Precisely, we study the extent to which Proposition \ref{no embedding}  remains valid for operator spaces. In fact, a slightly stronger, and more technical, result will be the main focus of this section. Its proof  follows from a simple gliding hump argument  (cf. \cite[Proposition 2.c.4]{LT1}). 

\begin{proposition}[Folklore]
Suppose $X$ and $(Y_i)_{i\in\N}$ be Banach spaces and assume that  $X$ is infinite dimensional and does not contain an isomorphic copy of $c_0$. If all operators $X\to Y_i$, for $i\in\N$, are strinctly singular, then   $X$ does not  isomorphically embed into $(\bigoplus_i Y_i)_{c_0}$.\label{no embedding.revisited}
\end{proposition}
 
 As a warm up for what is to come, we start with an elementary proposition:

 \begin{proposition}\label{compact operator into}
 Suppose $X$ and $(Y_i)_{i\in\N}$  are operator spaces such  that $X$ has no subspace isomorphic to $c_0$ and  every  completely bounded map from $X$ to $Y_i$, $i\in\N$,  is strictly singular. Then $X$ does not completely isomorphically embed into $(\bigoplus_i Y_i)_{c_0}$.
 \end{proposition}
 
 As the expert will notice, the proof of Proposition \ref{compact operator into} is essentially the one of Proposition \ref{no embedding.revisited}. We point out however that Proposition \ref{compact operator into} is \emph{not} the operator space version of Proposition \ref{no embedding.revisited}. Indeed, a truely operator space version would only assume that all operators $X\to Y_i$ are \emph{completely strictly singular}, meaning that their restrictions to infinite subspaces of $X$ are not complete isomorphic embeddings (but they can be isomorphic embeddings).

\begin{proof} [Proof of Proposition \ref{compact operator into}]
Suppose, for the sake of contradiction, that there exists a complete isomorphic embedding $u : X \to (\bigoplus_i Y_i)_{c_0}$.
 Write $u = \oplus_i u_i$; so, the hypothesis imply that each  $u_i : X \to Y_i$ is strictly singular.
A standard gliding hump argument produces a normalized   sequence $(x_j)_j \subset X$, for which there exists a sequence $(I_j)_j$ of intervals of $\N$ such that $\max(I_j)<\min(I_{j+1})$ and  $\|u x_j - Q_j u x_j\| < 4^{-j}$ for all $j\in\N$, where each  $Q_j$ denotes the canonical projection   $(\bigoplus_i Y_i)_{c_0}\to (\bigoplus_{i \in I_j} Y_i)_{c_0}$.
In particular, it follows that $(u x_j)_j$ is equivalent to the $c_0$-basis, hence the same must be true for $(x_j)_j$, which is impossible since $X$ does not contain an isomorphic copy of $c_0$.
\end{proof}

From this we deduce:

\begin{corollary}\label{into max}
 Let $(Y_i)_i$ be a sequence of operator spaces all of which are completely isomorphic to some Hilbert space endowed  with its maximal operator space structure.   Then any homogeneous Hilbertian subspace of $(\bigoplus_i Y_i)_{c_0}$ has the maximal operator space structure.
\end{corollary}

\begin{proof}
If a homogeneous Hilbertian $X$ is not maximal, then any c.b.~map from $X$ into some $Y_i$ is strictly singular. Indeed, fix $i\in\N$ and let $H$ be a Hilbert space so that $Y_i$ is competely isomorphic to $\MAX(H)$. Suppose there is a c.b.\ map  $u:X\to Y_i$ which is not strictly singular; so, there is an infinite  dimensional $Z\subset X$ such that $u\restriction Z$ is an isomorphic embedding. As $X$ is homogeneous Hilbertian, $Z$ is completely isomorphic to $X$. So, there is a c.b.\ map $v:X\to \MAX(H)$ which is also an isomorphic embedding. By the maximality of $\MAX(H)$,   $v$ must be a complete isomorphic embedding. This gives us a contradiction since every subspace of $\MAX(H)$ is also maximal. 

The result now follows from  Proposition \ref{compact operator into}.
\end{proof}

\subsection{Embeddings into $c_0$-sums of a single operator space} In this subsection, we prove Theorems \ref{Thm.HomHilb.SingleSpace} and \ref{Thm.Counterexample}. The next subsection will focus on embeddings into the $c_0$-sum of a sequence of operator spaces.

\begin{proof}[Proof of Theorem \ref{Thm.HomHilb.SingleSpace}]
Let $u\colon X\to c_0(Y)$ be  a complete  isomorphic embedding and, for each $i\in\N$, let $u_i\colon X\to Y$ be the composition of $u$ with the canonical projection $c_0(Y)\to Y$ onto the $i$-th coordinate of $c_0(Y)$.  As $X$ has infinite dimension, so does  $Y$. Hence,   $\mathrm{dens}(Y)=\mathrm{dens}(c_0(Y))$, which in turn implies that  $\mathrm{dens}(X)\leq \mathrm{dens}(Y)$. Therefore, as both $X$ and $Y$ are Hilbertian, we can assume that $X\subset Y$ as vector spaces and that there is $L\geq 1$ such that \[L^{-1}\|x\|_X\leq \|x\|_Y\leq L\|x\|_X\] for all $x\in X$. Moreover, to simplify notation, we fix  infinite sets $\Gamma$ and $\Lambda$ with $\Gamma\subset \Lambda$ and assume that  $X=\ell_2(\Gamma)$ and $Y=\ell_2(\Lambda)$ as vector spaces. Let $(e_\lambda)_{\lambda\in \Lambda}$ be the standard unit basis of $\ell_2(\Lambda)$.

\begin{case1}
The formal inclusion $X\hookrightarrow Y$ is completely bounded.
\end{case1}

Fix $\theta\geq 1$ such that $Y$ is $\theta$-homogeneous. If the inverse of the inclusion $X\hookrightarrow Y$ is also completely bounded, then $X\hookrightarrow Y$  is a complete isomorphic embedding and the  conclusion follows. If not,  there are $k\in\N$ and matrices $(a_\gamma)_{\gamma\in \Gamma}$ in $\M_k$ such that
\begin{itemize}
\item  $\|\sum_{\gamma\in \Gamma} a_\gamma \otimes e_
\gamma\|_{\M_k(Y)}=1$ and
\item  $\|\sum_{\gamma\in \Gamma} a_\gamma \otimes e_\gamma\|_{\M_k(X)}> \theta L\|u\|\|u^{-1}\|_{\cb}$. 
\end{itemize}

For each $i\in \N$, let $v_i\colon Y\to Y$ be the linear map defined by 
\[v_i(e_\lambda)=\left\{\begin{array}{ll}
u_i(e_\lambda),& \text{ if } \lambda\in \Gamma\\
0, &\text{ if } \lambda\in \Lambda\setminus \Gamma.
\end{array}\right.\]
 Clearly, $\|v_i\|\leq L\|u_i\|\leq L\|u\|$ for all $i\in\N$. By our choice of $\theta$, we have that \[\sup_{i\in \N}\|v_i\|_{\cb}<\theta L \|u\|.\] Therefore, it follows that
\begin{align*}
    \Big\|\sum_{\gamma\in \Gamma} a_\gamma \otimes u(e_\gamma)\Big\|_{\M_k(c_0(Y))} &= \sup_{i\in\N} \Big\|\sum_{\gamma\in \Gamma} a_\gamma \otimes u_i(e_\gamma)\Big\|_{\M_k(Y)}  \\
    &= \sup_{i\in\N} \Big\|\sum_{\gamma\in \Gamma} a_\gamma \otimes v_i(e_\gamma)\Big\|_{\M_k(Y)} \\
    &\leq \sup_{i\in \N} \|v_i\|_{\cb} \Big\|\sum_{\gamma\in \Gamma} a_\gamma \otimes e_\gamma\Big\|_{\M_k(Y)}\\
    &\leq \theta L\|u\|.
\end{align*}  However, as    \[\Big\|\sum_{\gamma\in \Gamma} a_\gamma \otimes e_\gamma\Big\|_{\M_k(X)}\leq \|u^{-1}\|_{\cb}\Big\|\sum_{\gamma\in \Gamma} a_\gamma \otimes u(e_\gamma)\Big\|_{\M_k(c_0(Y))} ,\] this contradicts our choice of $(a_\gamma)_{\gamma\in \Gamma}$.

\begin{case2}
Case 1 does not hold.
\end{case2}

We start by noticing that for each $i\in\N$ and each finite codimensional subspace $Z\subset X$, the restriction  $u_i\restriction _Z\colon Z\to u_i(Z)$ is not an isomorphism. Indeed, suppose this is not the case and fix offenders, say $i\in\N$ and $Z\subset X$. As $X$ is a homogeneous Hilbertian space and $Z$ has finite codimension in $X$, there is a complete isomorphism $v\colon X\to Z$. Clearly, the basic sequences $(u_i(v(e_\gamma)))_{\gamma\in \Gamma}$ and $(e_\gamma)_{\gamma\in \Gamma}$ are in $Y$, hence, as $Y$ is a homogeneous Hilbertian space there is a complete isomorphic embedding  $w\colon   u_i(Z)\to Y$ such that $ w(u_i(v(e_\gamma)))=e_\gamma$, for all $\gamma\in \Gamma$.   Since \[ w\circ u_i\circ v \colon X\to Y \] is the inclusion, this inclusion  must be completely bounded. This is a  contradiction since we assume Case 1 does not hold. 

Since $u_i\restriction _Z\colon Z\to u_i(Z)$ is not an isomorphism for all $i\in\N$ and all finite codimensional $Z\subset X$,   a standard gliding hump argument from Banach space theory gives that the basis of $\ell_2$ and $c_0$ are equivalent; contradiction.
\end{proof}

 The next result shows that homogeneity is necessary for Theorem \ref{Thm.HomHilb.SingleSpace} to hold.

 \begin{proof}[Proof of Theorem \ref{Thm.Counterexample} ]
 Let $Y=\MIN(\ell_2)\oplus\MAX(\ell_2)$ and fix a partition   $(S_k)_k$ of $\N$ into finite subsets with $\lim_k|S_k|=\infty$. Let $(e_i)_i$ be the canonical unit basis of $\ell_2$ and, given $x=\sum_ia_ie_i\in \ell_2$ and $k\in \N$, we let \[x\restriction _{S_k}=\sum_{i\in S_k}a_ie_i\in \ell_2(S_k).\] We then let $X$ be the operator space consisting of $\ell_2$ with the operator space structure given by the isometric embedding
 \[x\in \ell_2\mapsto (x, (x\restriction_{S_k})_k)\in \MIN(\ell_2)\oplus\Big(\bigoplus_k \MAX(\ell_2(S_k))\Big)_{c_0}.\]
Since $\MIN(\ell_2)\oplus(\bigoplus_k \MAX(\ell_2(S_k)))_{c_0}$ completely isometrically embeds into $c_0(Y)$,  it is clear that  $X$  completely isometrically embeds into $c_0(Y)$. 
 
 We are  left to notice that  $X$ does not completely isomorphically embed into $Y$. Suppose for a contradiction that such embedding $u\colon   X\to Y$ exists. Let $p\colon  Y\to \MIN(\ell_2)$ and $q\colon Y\to \MAX(\ell_2)$ denote the canonical projections. For each $k\in\N$, let \[X_k=\mathrm{span}\{e_i\in X\mid i\in S_k\}.\] So, $X_k=\MAX(\ell_2(S_k))$ completely isometrically. For each $k\in \N$, let $v_k=p\circ u\restriction_{X_k}$ and let \[v_{k,k}\colon \M_{1,k}(X_k)\to \M_{1,k}(\MIN(\ell_2))\] be its $1$-by-$k$ amplification.\footnote{If $E$ is an operator space, then $\M_{1,k}(E)$ denotes the subspace of $\M_k(E)$ consisting of all operators whose only nonzero rows are their first one. The space $\M_{k,1}(E)$ is defined similarly, but with the word ``columns'' substituting ``rows''. }

\begin{claim}
The maps $(v_{k,k})_k$ are  not  equi-isomorphisms. 
\end{claim}

\begin{proof}
If $(v_{k,k})_k$ are equi-isomorphisms, then so are $(v_{k})_k$. For each $k\in\N$, let  $w_k\colon   v_k(X_k)\to \MIN(\ell_2(S_k))$ be the linear map defined by $w_k(v_k(e_i))=e_i$ for all $i\in S_k$. As $\MIN(\ell_2)$ is homogeneous and $(v_k)_k$ are equi-isomorphisms,   $(w_k)_k$ are equi-complete isomorphisms. Therefore, letting \[w_{k,k}\colon  \M_{1,k}(v_k(X_k))\to \M_{1,k}( \MIN(\ell_2(S_k)))\] be the $1$-by-$k$ amplification of $w_k$, we obtain that the maps $(w_{k,k}\circ v_{k,k})_k$ are equi-isomorphisms.  However, $w_{k,k}\circ v_{k,k}$ is precisely the identity  \[\M_{1,k}( \MAX(\ell_2(S_k)))\to \M_{1,k}( \MIN(\ell_2(S_k))).\] This is a contradiction since those identities are not equi-isomorphisms. 
\end{proof}

As $(v_{k,k})_k$ are equi-completely bounded,   the previous claim implies that  there are a sequence $(n_k)_k$ in $\N$ and a sequence $(x_k)_k$ such that \begin{itemize}
\item $x_k\in \M_{1,n_k}(X_{n_k})$ and  $\|x_k\|_{\M_{1,n_k}(X_{n_k})}=1$, and

\item  $\delta=\inf_k\|q(u(x_k))\|_{\M_{1,n_k}(\MAX(\ell_2))}>0$.
\end{itemize}
Therefore, since  the $1$-by-$k$ amplifications of the identity $\MAX(\ell_2)\to R$ are isometries, by letting    $N_k=n_1+\ldots +n_k$, we have that
\[\|\begin{bmatrix}q(u(x_1))& \ldots &q(u(x_k))
\end{bmatrix}\|_{\M_{N_k}(\MAX(\ell_2))}\geq \delta k^{1/2}\]
for all $k\in\N$. On the other hand, as each $x_k$ is in $X_k$, it follows that
\[\|\begin{bmatrix}x_1& \ldots &x_k
\end{bmatrix}\|_{\M_{N_k}(X)}=\|\begin{bmatrix}x_1& \ldots &x_k
\end{bmatrix}\|_{\M_{N_k}(\MIN(\ell_2))}=1\]
for all $k\in\N$. It then follows that $\delta k^{1/2}\leq \|u\|_{\cb}$ for all $k\in\N$;
 contradiction. 
 \end{proof}
  
  \subsection{Embeddings into the $c_0$-sum of a sequence of operator spaces}
We now move to study what happens with Theorems \ref{Thm.HomHilb.SingleSpace} and \ref{Thm.Counterexample} if one replaces $c_0(Y)$ by $(\bigoplus_{n\in\N} Y_n)_{c_0}$.

\begin{proof}[Proof of Theorem \ref{Thm.Prop.HomHilb.MultipleSpace}]
For each $m\in\N$, let $R^{[m]}$ be $\ell_2$ as a Banach space, equipped with the operator space structure given by 
\[
\|x\|_{B(\ell_2) \otimes R^{[m]}} = \sup_p \|(I \otimes P)x\|_{B(\ell_2) \otimes R} ,
\]
where the supremum runs over all projections of rank $m$. For each $n\in\N$, let $Y_n=R^{[2^n]}$. So,    $(Y_n)_n$ are all completely isomorphic to each other, and to $\MIN(\ell_2)$ (although the isomorphism constants are not uniformly bounded). Moreover, since $R$ is $1$-homogeneous, it is clear that each $Y_n$ is also $1$-homogeneous. 

Define $X$ as the image of $\ell_2$ under the isometry  $u \colon  \ell_2 \to (\bigoplus_n Y_n)_{c_0} $ given by $u( \xi)= (\xi/n)_{n=1}^\infty$ for all $\xi\in \ell_2$.  So, $X$ is clearly Hilbertian and, as each $Y_n$ is $1$-homogeneous, $X$ is also $1$-homogeneous.  We are left to notice that   $X$ does not completely isomorphically embeds into  $\MIN(\ell_2)$. Since $\MIN(\ell_2)$ is homogeneous, it is enough to show that the identity $X\to \MIN(\ell_2)$ is not a complete isomorphism. For that, let $(e_n)_n$ be the canonical basis of $\ell_2$ and for each $n\in\N$ let 
\[x_n=\begin{bmatrix}
e_1&\ldots &e_{2^n}
\end{bmatrix}\in \M_{1,2^n}(\MIN(\ell_2)).\]
So, $\|x_n\|_{\M_{1,2^n}(\MIN(\ell_2))}=1$ but $
\|x_n\|_{\M_{1,2^n}(X)} \geq   2^{n/2}/n $ for all $n\in\N$.
\end{proof}

 We are left to prove Theorem \ref{Thm.fix problem}. For that, we will now we turn our attention from embeddings to quotients. But first, we recall a definition:   an operator $T :X\to Y$ between Banach spaces  is called \emph{strictly cosingular} if for all   infinite dimensional Banach spaces $Y_0$ and all   operators $q : Y \to Y_0$, we have that  $q T$ not surjective. Strictly cosingular operators form an ideal; see e.g.~\cite[Section 3.4]{Aiena} for more information.

\begin{proposition}\label{no quotient cosingular}
Suppose $X$ and $(Y_i)_{i\in\N}$   are operator spaces so $\dim(X)=\infty$ and  any c.b.~map from $Y_i$ to $X$, $i\in\N$, is strictly cosingular. If there exists a c.b.~surjection from $(\bigoplus Y_i)_{c_0}$ onto $X$, then $X$ contains a copy of $\MIN(c_0)$.
\end{proposition}

Before proving Proposition \ref{no quotient cosingular}, we need a lemma. The following easy lemma is known to experts, but we provide a proof of it for the readers convinience. 

\begin{lemma}\label{l:perturb surjection}
For every surjection $T : F \to E$ between Banach spaces, there exists $\delta > 0$ so that $T + S$ is surjective for any operator $S :F\to E$  with $\|S\| < \delta$.
\end{lemma}

\begin{proof}
It is well known that an operator $U$ is surjective if and only if its adjoint is bounded below. Thus, we can find $\delta > 0$ so that $\|T^* e^*\| \geq \delta \|e^*\|$ for any $e^* \in E^*$. If $\|S\| < \delta$, then
$$
\|(T+S)^* e^*\| \geq \|T^* e^*\| - \|S\| \|e^*\| \geq (\delta - \|S\|) \|e^*\| 
$$
holds for any $e^* \in E^*$, showing the surjectivity of $T+S$.
\end{proof}

\begin{proof}[Proof of Proposition \ref{no quotient cosingular}]
For the sake of convenience, write $Z = (\bigoplus_i Y_i)_{c_0}$ and, for each $n \in \N$, let $Z_n = \bigoplus_{i \leq n} Y_i$ and $Z^n = \bigoplus_{i > n} Y_i$; as usual, we view $Z^n$ and $Z^n$ as  subspaces of $Z$ in the canonical way. Suppose $u : Z \to X$ is a completely bounded surjection. Then $\inf_n \|u|_{Z^n}\| > 0$. Indeed, otherwise, by Lemma \ref{l:perturb surjection}, $u - u\restriction_{Z^n} = u\restriction_{Z_n}$ is surjective for $n$ large enough. However, being a finite sum of strictly consingular operators, $u\restriction_{Z_n}$ is also strictly cosingular and, in particular, it cannot be surjective since $\dim(X)=\infty$.

A gliding hump argument then produces an increasing sequence $(n_j)_j$ of naturals and a sequence $(z_j)_j$ of normalized vectors such that   $z_j \in \bigoplus_{n_{j-1} < i \leq n_j} Y_i$, for all $j\in\N$, and  $\inf_j \|u z_j\| > 0$. Clearly, $(z_j)_j$ is weakly null, hence so is $(u z_j)_j$. By \cite[Proposition 1.5.4]{AlbiacKaltonBook}, going to a subsequence if necessary, we can assume that $(u z_j)_j$ is a basic sequence.  For simplicity of notation, let $x_j=uz_j$ for all $j\in\N$.

For any $(a_k)_k \subset B(H)$, we have \[\Big\|\sum_k a_k \otimes z_{j_k}\Big\|_{B(H) \otimes Z} \leq \sup_k \|a_k\|.\] Therefore, we have that  \[\Big\|\sum_k a_k \otimes x_k\Big\|_{B(H) \otimes X} \leq \|u\|_{\cb} \sup_k \|a_k\|.\] On the other hand, since $(x_k)_k$ is a basic sequence, one can find functionals $x_k^* \in X^*$, biorthogonal to $x_k$'s\footnote{I.e., $x_k^*(x_k)=1$ and $x^*_k(x_j)=0$ for all $k\neq j$.} and such that  $\sup_k \|x_k^*\| < \infty$. As for rank one operators the operator and c.b.~norms coincide, we have that 
\[\|a_m\|   \leq \|x_m^*\| \Big\|\sum_k a_k \otimes x_k\Big\|_{B(H) \otimes X}\]
for all $m\in\N$. Thus, $(x_k)_{k \in \N}$ spans a copy of $\MIN(c_0)$ in $X$.
\end{proof}

Note that, under the hypothesis of Proposition \ref{no quotient cosingular}, $X$ cannot embed completely complementably in $(\bigoplus_i Y_i)_{c_0}$. Moreover, we have:

\begin{corollary}\label{c:row_col_doesnt_embed}
Suppose $X_0$ is either $R$ or $C$, and $X$ is a Hilbertian operator space for which there exists a completely bounded surjection from $X$ onto $X_0$. 
Further, let $(Y_i)_{i\in\N}$ be   operator spaces   such that any completely bounded operator from $Y_i$ to $X_0$, $i\in\N$, is strictly cosingular. Then $X$ does not completely isomorphically embed into $(\bigoplus_i Y_i)_{c_0}$.
\end{corollary}

\begin{proof}
Suppose, for the sake of contradiction, that $(\bigoplus_i Y_i)_{c_0}$ contains a subspace $X'$ completely isomorphic to $X$. Find a completely bounded surjection $v : X' \to X_0$. 
By the injectivity of $X_0$, $v$ extends to a c.b.~surjection $\widetilde{v} : (\bigoplus_i Y_i)_{c_0} \to X_0$. However, such maps cannot exist, by Proposition \ref{no quotient cosingular}.
\end{proof}
 
 We highlight a  corollary of the previous corollary:

 \begin{proof}[Proof of Theorem \ref{Thm.fix problem}]
 This follows from Corollary \ref{c:row_col_doesnt_embed} since if $X_0\in \{R,C\}$ we have that any completely bounded operator $\MIN(H)\to X_0$ is compact and, in particular, strictly cosingular. Indeed, suppose $u:E\to F$ is a compact operator and $q:F\to Y_0$ is a surjective operator with $\dim(Y_0)=\infty$. Then, the open mapping theorem implies that the induced map  $ \widetilde{qu}:F/\ker(qu)\to Y_0$  is an isomorphism. However this is impossible since $\widetilde{qu}$ is compact and $\dim(Y_0)=\infty$. 
 \end{proof}
  
 \begin{remark}\label{Remark1}
A version of Theorem \ref{Thm.fix problem} was proven in the proof of \cite[Theorem 4.2]{BragaChavezDominguez2020PAMS}. However, this proof   contains a gap in its last paragraph. Theorem \ref{Thm.fix problem} fixes this gap and \cite[Theorem 4.2]{BragaChavezDominguez2020PAMS} remains valid with no changes in its statement. We also point out that, fixing the gap of \cite[Theorem 4.2]{BragaChavezDominguez2020PAMS} is much easier and requires only some small modifications in its proof. But we chose to present this different proof here since it leads to a more general result (Proposition \ref{no quotient cosingular})  and applications (Corollaries \ref{c:row_col_doesnt_embed} and  \ref{c:min vs max}).
\end{remark}

We finish this paper with another application of Proposition \ref{no quotient cosingular}:

\begin{corollary}\label{c:min vs max}
Suppose $E$ and $(F_i)_{i\in\N}$ are  Banach spaces, with $\dim E = \infty$. Let $(Y_i)_{i\in\N}$ be a sequence of operator spaces   such that  $Y_i$ is completely isomorphic to $\MIN(F_i)$, $R$, or $C$. Then there is no c.b.~surjection from $(\bigoplus_i Y_i)_{c_0}$ onto $\MAX(E)$.
\end{corollary}

\begin{proof}
Note that $\MIN(c_0)$ (just as any minimal space) is exact, hence, by \cite[Corollary 2.9]{Oikhberg2004}, $\MAX(E)$ does not contain a completely isomorphic copy of $\MIN(c_0)$. To apply Proposition \ref{no quotient cosingular}, one therefore has to verify that, for each $i$, every completely bounded operator  $Y_i\to \MAX(E)$ is strictly cosingular.

First assume $Y_i$ is completely isomorphic to $\MIN(F_i)$. By \cite[Section 4]{Paulsen96}, $\CB(Y_i, \MAX(E))$ coincides with $\Gamma_2^*(F_i,E)$; this ideal consists of operators $T : F_i \to E$ for which there exists a factorization $T = vu$, with $u \in \Pi_2(F_i, H)$ ($H$ is a Hilbert space) and $v \in \Pi_2^*(H,E)$ (that is, $v^*$ is $2$-summing); see \cite[Chapter 7]{DJT} for more information. Therefore, we have to show that, for any $E$ and $F$, any element of $\Gamma_2^*(F,E)$ is strictly cosingular. For that,   fix $T \in \Gamma_2^*(F,E)$. By the preceding discussion, $T^*$ is $2$-summing. If $q : E \to E_0$ is a quotient map and   $\dim E_0 = \infty$, then $T^* q^*$ is $2$-summing as well, hence not an isomorphic injection. Thus, $q T$ is not a surjection.

Now suppose $Y_i$ is completely isomorphic to either $C$ or $R$. By \cite[Proposition 5.11]{Pisier-OH}, any c.b.~map from $\MIN(G)$,  for any operator space $G$, to $R$ or $C$ is $2$-summing. If $T : Y_i \to \MAX(F)$ is c.b., then $T^*$ is $2$-summing. Then $T^*$ is not an isomorphism on any infinite dimensional subspace; this establishes the strict cosingularity of $T$.
\end{proof}

\end{document}